\documentclass[a4paper,leqno]{mathscan}

\usepackage[all]{xy}
\usepackage{graphicx}
 \usepackage{verbatim}
\usepackage[danish,english]{babel}
\usepackage{color}
\usepackage{enumerate}
\usepackage[abs]{overpic}

\newcommand{\R}{\mathbb R}
\newcommand{\N}{\mathbb N}

\newcommand{\union}{\cup}

\newcommand{\cisom}[1]{[#1]_{\tiny{isom}}}
\newcommand{\Laff}{L_{\text{aff}}}

\newtheorem{thm}{Theorem}[section]
\newtheorem{pro}[thm]{Proposition}
\newtheorem{cor}[thm]{Corollary}
\newtheorem{lem}[thm]{Lemma}

\theoremstyle{definition}

\newtheorem{defn}[thm]{Definition}  

\usepackage[hidelinks]{hyperref}
\begin{document}

\title[Continuous hull]
 {Construction of the continuous hull for the combinatorics of a  regular pentagonal tiling of the plane}

\author{Maria Ramirez-Solano}
\thanks{}

\address{Department of Mathematics, University of Copenhagen, Universitetsparken 5,
2100 K\o benhavn \O ,
 Denmark.}

\email{mrs@math.ku.dk}



\thanks{Supported by the Danish National Research Foundation through the Centre
for Symmetry and Deformation (DNRF92), and by the Faculty of Science of the University of Copenhagen.}

\begin{abstract}
In \cite{MRSdiscretehull} we gave the construction of a discrete hull for a combinatorial pentagonal tiling of the plane.
In this paper, we give the construction of a continuous hull for the same combinatorial tiling.
\end{abstract}

\maketitle
\section{Introduction}
In this paper, we consider "a regular pentagonal tiling of the plane" which comes from \cite{StephensonBowers97}, and we show it in Figure \ref{f:conformalregularpentagonaltiling}.
The combinatorics of this tiling are constructed using the subdivision rule shown in Figure \ref{f:subdivisionmapnewdecorated}, while the geometry comes from the theory of Circle Packing.
Although a beautiful tiling, we showed in \cite{MRSnonFLCpentTiling} that this tiling is not FLC with respect to the set of conformal isomorphisms.
This prevents us from studying the tiling with the standard theory for Euclidean tilings of the plane.
In such theory, two compact topological spaces, known as the continuous hull and the discrete hull of the tiling, are constructed.
In \cite{MRSdiscretehull} we stripped down the geometry from the pentagonal tiling and kept only the combinatorial part, and we named it "the combinatorial pentagonal tiling $K$". We defined a \emph{combinatorial tiling} as a 2-dimensional CW-complex homeomorphic to the plane.
 In the same article \cite{MRSdiscretehull}, we showed a construction of a discrete hull for the combinatorial pentagonal tiling $K$, which is a cantor space.
 Elements of the discrete hull are isomorphism classes of combinatorial tilings with a distinguished vertex (the origin of the tiling) locally isomorphic to $K$. Imposing a simpler geometry on the combinatorics, we will now construct its continuous version, known as the continuous hull $\Omega$ of $K$.
To this end, we extend the edge-metric of each tiling of the discrete hull and then extend the discrete hull's  metric to $\Omega$.
We will also define a substitution map on $\Omega$.  We make distinctions between elements of $\Omega$ to the level of points, hence the use of the word continuous in the name. Such metric space is compact (Theorem \ref{t:ContinuousHullCompact}), and unlike its discrete counterpart, this substitution map is a homeomorphism (Theorem \ref{t:OmegaHomeomorphism}).
The extension is of course not unique, but it is a natural one, and we did several attempts to obtain it by playing around with possible topologies preserving the cellular structure, and possible metrics. We like our definition of $\Omega$ for it resembles in many ways the continuous hull of Euclidean tilings of the plane. In Theorem \ref{t:dequivalenttodponOmega}, we show that the discrete metric is equivalent to the continuous one restricted to the discrete hull.
Hence, metric related properties of $\Xi$ are automatically preserved under the metric of $\Omega$.
At the end of this paper, we construct an equivalence relation $R$ on the continuous hull.  It is an open question whether the $C^*$-algebra $C^*(R)$ exists. For that it needs a topology and a Haar system.

\begin{figure}[htbp]
  \begin{minipage}[b]{0.48\linewidth}
    \centering
    \includegraphics[width=\linewidth]{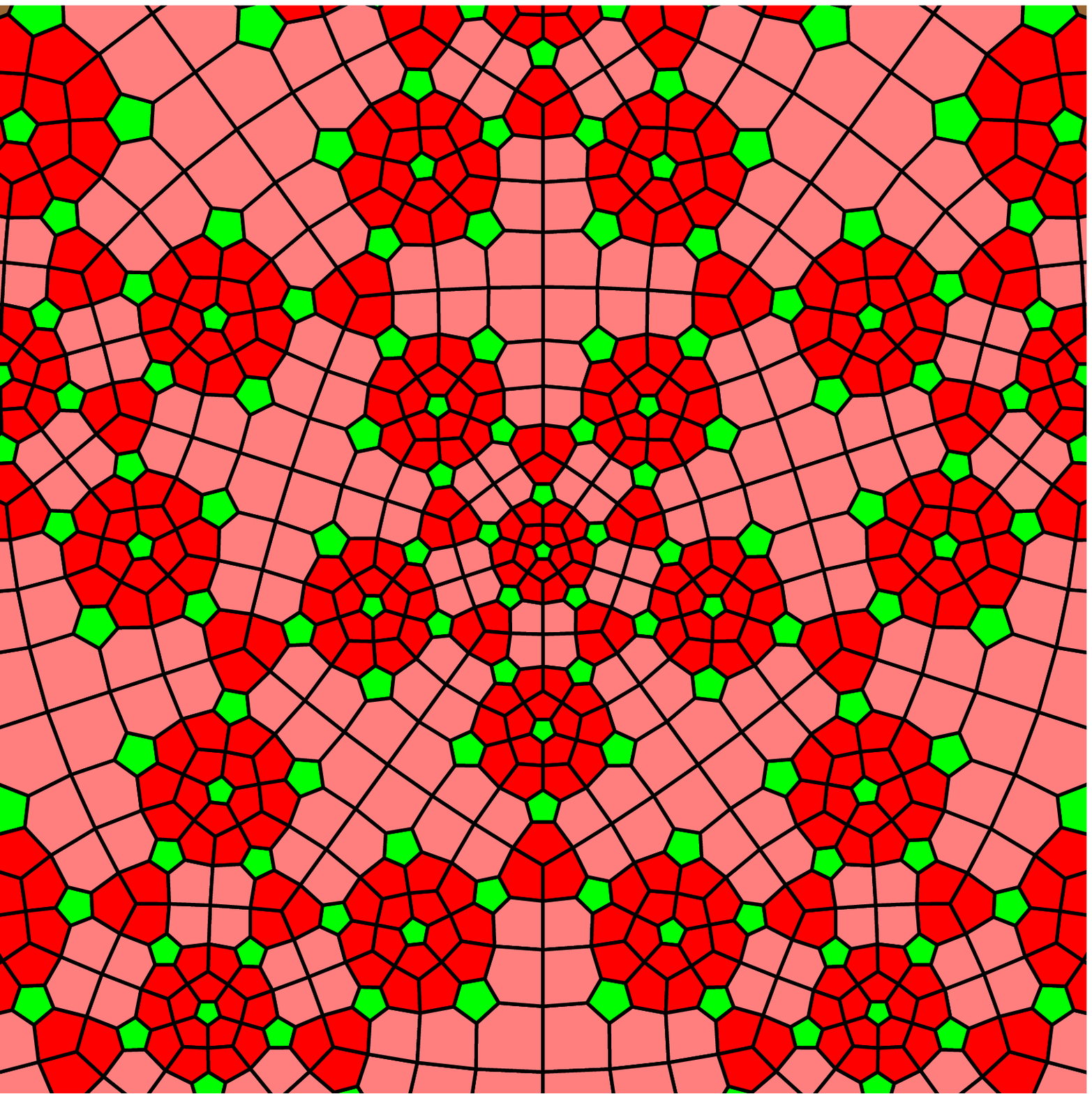}
    \caption{A conformal regular pentagonal tiling of the plane.}
    \label{f:conformalregularpentagonaltiling}
  \end{minipage}
  \begin{minipage}[b]{0.48\linewidth}
    \centering
    \includegraphics[width=\linewidth,viewport=150 150 450 450,clip]{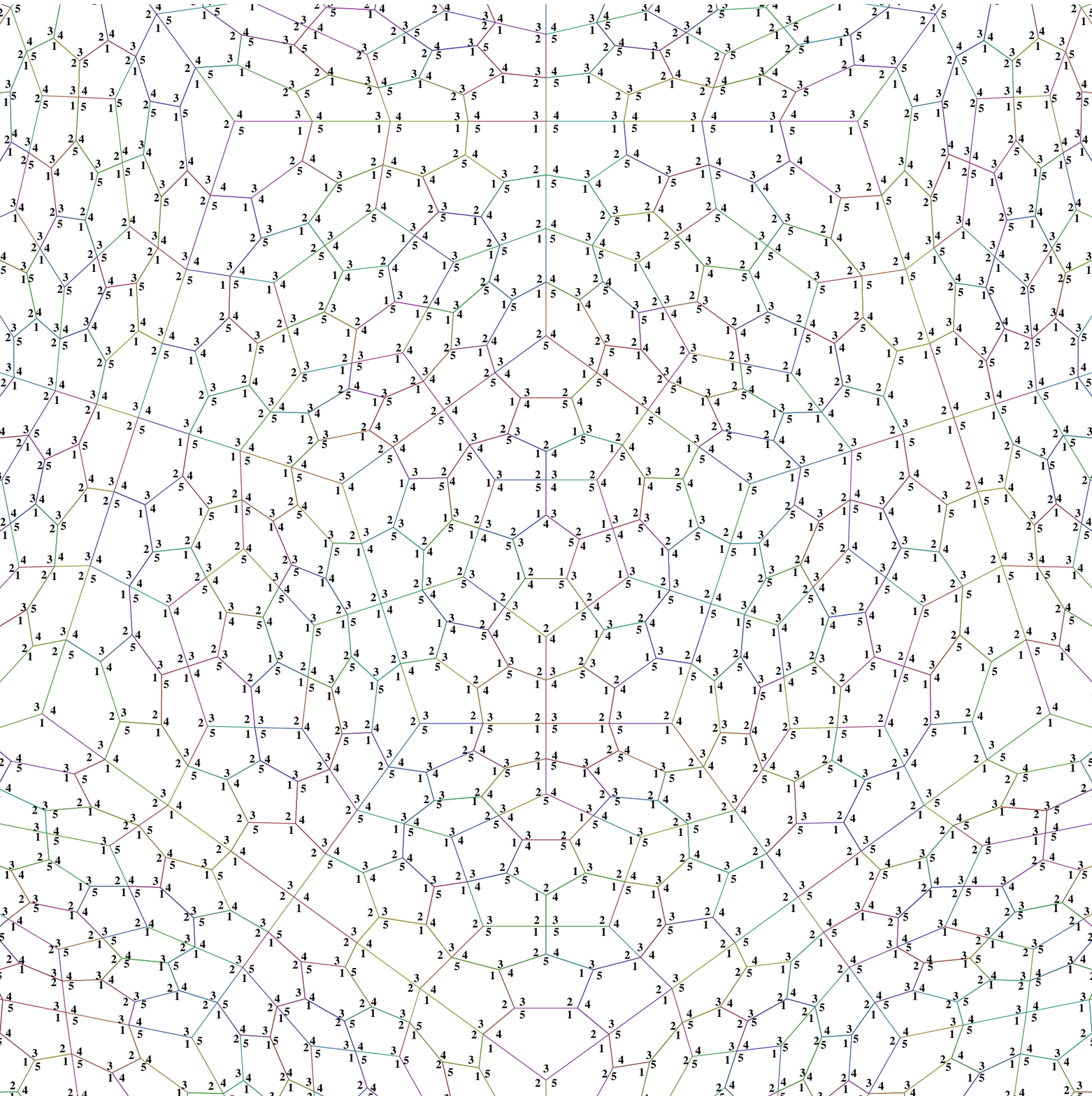}
    \caption{The decorated combinatorial tiling $K$.}
    \label{f:decoratedtiles}
  \end{minipage}
\end{figure}

\begin{figure}
  \centering
  \includegraphics[scale=1]{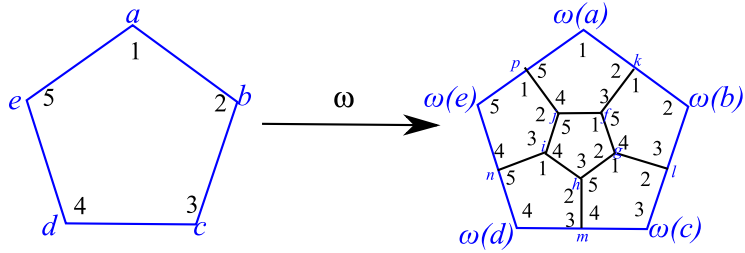}\\
  \caption{Subdivision map. }\label{f:subdivisionmapnewdecorated}
\end{figure}

We start by restating some definitions and results from \cite{MRSdiscretehull}.
The combinatorial pentagonal tiling $K$ is constructed as follows.
Define $K_0$ as a combinatorial pentagon, which is a space homeomorphic to the closed unit disk with five distinguished points on its boundary.
Define $K_n:=\omega^n(K_0)$, $n\in\N_0$. Every $K_n$ has a distinguished central pentagon. We define $\iota_n:K_n\to K_{n+1}$ as an embedding which maps the central pentagon of $K_n$ to the central pentagon of $K_{n+1}$.  Define the complex $$K:=\lim_{n\to\infty} K_n,$$
as the direct limit of the sequence of the finite CW-complexes $K_n$ and embeddings $\iota_n$.
We say that a combinatorial tiling $L$ is \emph{locally isomorphic} to $K$ if for every patch $P$ of $L$ there is a patch $Q$ of $K$ such that $P$ and $Q$ are cell-preserving isomorphic. For such $L$, we define the ball \emph{$B(v,n,L)$} as the subcomplex  of $L$ whose tiles have the property that all its vertices are within edge-distance $n$ of the vertex $v\in L$.

\section{The metric space $(\Omega,d)$}
We start by equipping every combinatorial tiling $L$ locally isomorphic to $K$ with a piecewise affine structure, just as it was done for $K$ in \cite{StephensonBowers97}.
\begin{defn}[$L_{\text{aff}}$]
  Given a combinatorial tiling $L$ locally isomorphic to $K$, we define  \emph{$L_{\text{aff}}$} as the CW-complex $L$, such that the topological space is equipped with the following metric: First we define the unit-edge metric on the one skeleton $L^{(1)}$ making each edge isometric to the unit interval. Then we extend this metric to faces so that each face is isometrically  an Euclidean regular pentagon of edge-length one. We define the distance of an arbitrary pair of points as the length of the shortest path between them.
\end{defn}

  The resulting metric defines the original CW-topology of $L$. Indeed, if $U$ is open in $L$, and $x\in U$, then it is easy to construct a ball with tiny radius centered at $x$ and contained in $U$. On the other hand, if $B$ is a ball, and $x\in B$ then it is also easy to construct an open set contained in the ball $B$, as tiny metric balls are open sets. In particular, cells remain cells, and the automorphisms of $L$ are precisely those of $\Laff$. The piecewise affine structure of $\Laff$ is equivalent to pasting regular unit pentagons together in the pattern of  $L$.
We denote with \emph{$d'$} the \emph{unit-edge metric} on the 1-skeleton $\Laff^{(1)}$, and with \emph{$d$} the \emph{metric} on $\Laff$. The unit-edge metric $d'$ is simply the combinatorial metric of $L$.
A path on $L$ (or $\Laff^{(1)}$) will be called an edge-path or a $d'$-path.
A path on $\Laff$ will be called just a path or a $d$-path.
The restriction of $d$ to $\Laff^{(1)}$ induces a metric on $\Laff^{(1)}$ and the following lemma shows that these two metrics are equivalent.
\begin{thm}\label{t:dledple3donLaff}
  The discrete metric $d'$ on $\Laff^{(1)}$ is equivalent to the metric $d$ on $\Laff$ restricted to $\Laff^{(1)}$. More precisely,  $d\le d' \le 3d$.
\end{thm}
\begin{proof}
  First, $d\le d'$ because $d$ is the length of the shortest path between two points (and a shortest edge-path is not necessarily a shortest path).
  We will now show that $d'\le 3d$.  Let $\gamma$ be a path of shortest length between two vertices $u,v$. We need to construct an edge-path $\gamma'$ connecting these two vertices. We will show that this edge-path has length less than $3d$. Since $\gamma$ goes through pentagons, it has three possibilities: (1) It  goes along an entire edge. (2) It crosses two non-adjacent edges of a pentagon. (3) It crosses two adjacent edges of a pentagon. See Figure \ref{f:dprimele3d}.
  For case (1) and (2), we have $d'\le 3d$.  For case (3), we have $d'\le 2d$ by the law of cosines.
  Hence, creating an edge-path will consist in gluing pieces of type $(1)$, $(2)$ or $(3)$ sometimes removing some edges. In summary, the $d$-path $\gamma$ of shortest length  induces an edge-path $\gamma'$ of length at most $3d$. Since $d'$ is the length of shortest edge-path,   $d'(u,v)\le edgelength(\gamma')\le 3d(u,v)$.
\end{proof}
\begin{figure}
\centering
  \includegraphics[scale=.7]{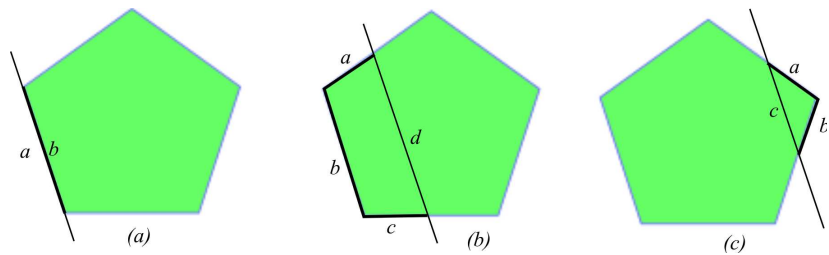}\\
  \caption{A path of shortest length replaced with an edge-path: $(a)$ $a=1=b$. $(b)$ $a+b+c\le 3\le 3d$. $(c)$ $a+b\le 2c$ by the law of the cosines.}\label{f:dprimele3d}
\end{figure}

\begin{cor}\label{c:dballcontainedindpball}
  Let $B_d(v,n,L)$ be a ball with respect to the metric $d$.
  Let $B_{d'}(v,n,L)$ be a ball with respect to the metric $d'$.
  Then $B_{d'}(v,n,L)\subset B_d(v,n,L)$ and $B_{d}(v,n,L)\subset B_{d'}(v,3n,L)$
\end{cor}
\begin{proof}
  If $u\in B_{d'}(v,n,L)$ then  $d(v,u)\le d'(v,u)\le n$ so $u\in B_d(v,n,L)$.
  If $u\in B_{d}(v,n,L)$ then  $d'(v,u)\le 3d(v,u)\le 3n$ so $u\in B_{d'}(v,3n,L)$.
\end{proof}

We define the continuous hull $\Omega$ as
\begin{center}
    \begin{tabular}{   @{}r@{}  @{}p{10cm}@{} }
   $\Omega:=\{[L,x]_{isom}\mid $ & $\, L$  is a combinatorial tiling locally isomorphic to $K$ and point  $x\in \Laff\}.$
    \end{tabular}
\end{center}
With $(L',x')\in [L,x]_{isom}$, we mean that there is an isomorphism $\phi:\Laff\to \Laff'$ which is cell-preserving, decoration preserving,  and isometric on each cell, and such that $\phi(x)=x'$.

We define the metric $d$ on $\Omega$ as follows.
\begin{defn}[metric $d$ on $\Omega$]\index{metric $d$ on $\Omega$}
Define $d:\Omega\times\Omega\to \R_+$ by
$$d(\cisom{L,x},\cisom{L',x'}):=min(\tfrac{1}{\sqrt2},\inf\Lambda),$$
where  $\Lambda\subset \R_+$, and $\varepsilon\in\Lambda$  if there exists maps
\begin{eqnarray*}
&& \phi:B(x,1/\varepsilon,L)\to L',\,\,\,\qquad  d_{L'}(\phi(x),x')\le \varepsilon\\
&& \phi':B(x',1/\varepsilon,L')\to L,\qquad d_{L}(\phi'(x'),x)\le \varepsilon
\end{eqnarray*}
 which are  cell-preserving maps, are isometries (hence homeomorphisms onto their images) preserve the decorations and degree of the vertices.
\end{defn}
A cell-preserving map extends to an isometry on cells.
An isometry on cells, does not necessarily extend to an isometry.
But it does on a small region. We show this result in the following lemma.

\begin{lem}\label{l:isometrywhenphirestricted}
Let $\phi:B(x,\varepsilon^{-1},L)\to L'$, $\phi':B(x',\varepsilon^{-1},L')\to L$ be  continuous cell-preserving maps such that the restriction to each cell is an isometry and $d(\phi(x),x')\le\varepsilon$ and $d(\phi'(x'),x)\le\varepsilon$ and $\phi\circ\phi'=id=\phi'\circ\phi$ on their intersection. Then $\phi$ restricted to $B(x,\varepsilon^{-1}/6,L)$ is an isometry.
\end{lem}
\begin{proof}
First notice that if $y,z\in B(x,\varepsilon^{-1},L)$ and $p$ is a path inside this ball from $y$ to $z$, then since $\phi$ is an isometry on each cell, $\phi(p)$ is a path from $\phi(y)$ to $\phi(z)$ of same length. However, $\phi(y)$ or $\phi(z)$ or the path is not necessarily inside the ball $B(x',\varepsilon^{-1},L')$. This is the reason why we look at a smaller ball.
Moreover, $B(\phi(x), \varepsilon^{-1}/6,L')\subset B(x', \varepsilon^{-1}/6+\varepsilon,L')\subset B(x',\varepsilon^{-1},L')$, as the distance between $\phi(x)$ and $x'$ is at most $\varepsilon\le 1/\sqrt2$.
  We claim that $\phi(B(x,\varepsilon^{-1}/6,L))=B(\phi(x),\varepsilon^{-1}/6,L')$.\\
     Indeed if $y\in B(x,\varepsilon^{-1}/6,L)$ then there is a path $p$ of length smaller than $\varepsilon^{-1}/6$ inside this ball from $y$ to $x$; since $\phi$ is an isometry on each cell, $\phi(p)$ is a path from $\phi(y)$ to $\phi(x)$ of same length. Hence $\phi(B(x,\varepsilon^{-1}/6,L))\subset B(\phi(x),\varepsilon^{-1}/6,L')$.\\
  Since  $B(\phi(x),\varepsilon^{-1}/6,L')\subset B(x',\varepsilon^{-1},L')$, $\phi'\phi(x)=x$.\\
  If $z\in B(\phi(x),\varepsilon^{-1}/6,L')$ then there is a path $p$ of length smaller than $\varepsilon^{-1}/6$ inside this ball from $z$ to $\phi(x)$; since $\phi'$ is an isometry on each cell, $\phi'(p)$ is a path from $\phi'(z)$ to $\phi'(\phi(x))=x$ of same length. Thus $\phi'(z)\in B(x,\varepsilon^{-1}/6,L)$. Thus we have found an element $\phi'(z)\in B(x,\varepsilon^{-1}/6,L)$ whose image $\phi(\phi'(z))$ is $z$. That is, $B(\phi(x),\varepsilon^{-1}/6,L')\subset \phi(B(x,\varepsilon^{-1}/6,L))$, and proves our claim.

  Let $y,z\in B(x,\varepsilon^{-1}/6,L)$. Then there is a path from $y$ to $z$ going through  $x$ of length at most $2(\varepsilon^{-1}/6)=\varepsilon^{-1}/3$. Thus, any path $p$ of minimal length from $y$ to $z$ is contained in the large ball $B(x,\varepsilon^{-1},L)$ because any path from $y$ to the boundary of this large ball is at least $\varepsilon^{-1}-\varepsilon^{-1}/6=5\varepsilon^{-1}/6$.
  Since $\phi$ is an isometry on each cell, $\phi(p)$ is a path from $\phi(y)$ to $\phi(z)$ of same length.
    Hence, $d(x,y)\ge d(\phi(x),\phi(y))$ and $\phi(y),\phi(z)\in\phi(B(x, \varepsilon^{-1}/6,L))=B(\phi(x), \varepsilon^{-1}/6,L')\subset B(x', \varepsilon^{-1}/6+\varepsilon,L')$.\\
  A path $q$ of minimal length between $\phi(y)$ and $\phi(z)$ is also contained in the large ball $B(x',\varepsilon^{-1},L')$ because the distance from $\phi(y)$ to the boundary of the large ball  $B(x',\varepsilon^{-1},L')$ is at least $\varepsilon^{-1}-(\varepsilon^{-1}/6+\varepsilon)$ which is greater than $\varepsilon^{-1}/3$ if $\varepsilon<1/\sqrt 2$.  Since $\phi'$ is an isometry on each cell, $\phi'( q)$ is a path from $\phi'\phi(y)$ to $\phi'\phi(z)$ of same length. Hence $d(\phi(x),\phi(y))\ge d(\phi'(\phi(x)),\phi'(\phi(y)))=d(x,y)$.

\end{proof}

\begin{figure}
\centering
  \includegraphics[scale=.9]{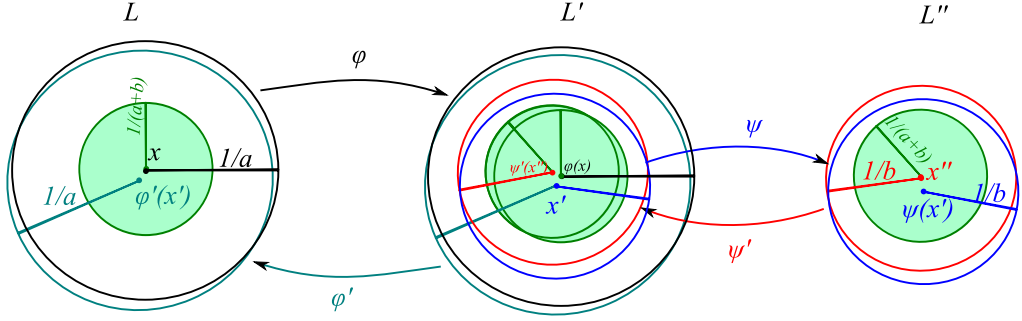}\\
  \caption{Triangle inequality.}\label{f:triangleinequalityOmega}
\end{figure}

\begin{lem}
  The map $d:\Omega\times \Omega\to \R_+$ is a metric.
\end{lem}
\begin{proof}
  The map $d$ is by definition symmetric and positive.
  If the distance $d(\cisom{L,x},\cisom{L',x'})=0$ then they both agree on balls of any radius $n>1$ centered at $x$ and $x'$ and the distance between the centers is at most $1/n$.
  That is, there exist maps $\phi_n:B(x,n,L)\to L'$, $\phi'_n:B(x',n,L')\to L$, which are  cell-preserving maps, are isometries, preserve the decorations and degree of the vertices, and such that $d(\phi_n(x),x')\le 1/n$, $d(\phi_n'(x'),x)\le 1/n$.
  Let $v$ be one of the closest vertices of $x$, and let $v'_n:=\phi_n(v)$.
  Since $d(v_n',x')\le d(v_n',\phi_n(x))+ d(\phi_n(x),x')\le 1.62+1/n$, and no two same tiles with exterior decoration are found next to each other, $v_n'$ is independent of $n$ so we let $v':=v_n'$.
  Since $\phi_n$ is cell-preserving and $d(x,v)<1.62$, the restriction  $\phi_n:B(v,n-2,L)\to L'$ is cell preserving with $\phi_n(v)=v'$.
  Similarly, since $\phi_n'$ is cell-preserving and $d(v',x')\le 1.62+1/n$, we have $d(\phi_n'(v'),x)\le d(\phi_n'(v'),\phi_n'(x'))+d(\phi_n'(x'),x) \le 1.62+2/n$ so $\phi_n'(v')=v$ and the restriction  $\phi'_n:B(v',n-2,L')\to L$ is cell-preserving.
  Thus $d'(\cisom{L,v},\cisom{L',v'})\le 1/(n-2)$, and so $d'(\cisom{L,v},\cisom{L',v'})=0$.
  Thus $\cisom{L,v}=\cisom{L',v'}$ combinatorially.
  Thus, $(\Laff,v)\cong (\Laff',v')$ and let $\Psi$ denote the isomorphism.  Since $\phi_n(x)$ converges to $x'$  and $\Psi$ restricted to a smaller ball
  $B(v,(n-2)/6,L)$ is $\phi_n$ for large $n$,
   $\cisom{L,x}=\cisom{L',x'}$.

  We will now show that $d$ satisfies the triangle inequality.
  Suppose the distance $d(\cisom{L,x},\cisom{L',x'})=a$, $d(\cisom{L',x'},\cisom{L'',x''})=b$.
  By definition, there exist isometries
  \begin{eqnarray*}
   && \phi:B(x,1/a,L)\to L',\quad d(\phi(x),x')\le a\\
   && \phi':B(x',1/a,L')\to L, \quad d(\phi'(x'),x)\le a\\
   && \psi:B(x',1/b,L')\to L'',\quad d(\psi(x'),x'')\le b\\
   && \psi':B(x'',1/b,L'')\to L',\quad d(\psi'(x''),x')\le b.
  \end{eqnarray*}
  Without loss of generality assume that $a\le b\le 1/\sqrt 2$. Furthermore, we can even assume $a+b\le 1/\sqrt 2 $ otherwise the triangle inequality automatically holds.
  We claim that
  \begin{eqnarray*}
   && \psi\phi:B(x,1/(a+b),L)\to L'',\quad d(\psi\phi(x),x'')\le a+b\\
   && \phi'\psi':B(x'',1/(a+b),L'')\to L'',\quad d(\phi'\psi'(x''),x)\le a+b.
  \end{eqnarray*}
  Hence $$d(\cisom{L,x},\cisom{L'',x''})\le d(\cisom{L,x},\cisom{L',x'})+d(\cisom{L',x'},\cisom{L'',x''}).$$ We now prove our claim.
  By the triangle inequality, we have $$d(\psi\phi(x),x'')\le d(\psi\phi(x),\psi(x'))+d(\psi(x'),x'')\le a+b.$$
  Also by the triangle inequality we have $d(\phi'\psi'(x''),x)\le d(\phi'\psi'(x''),\phi'(x'))+d(\phi'(x'),x)\le a+b$.
  We only need to show that $\phi(B(x,1/(a+b),L))\subset B(x',1/b,L')$ and that $\psi'(B(x'',1/(a+b),L''))\subset B(x',1/a,L')$. See Figure \ref{f:triangleinequalityOmega}.
  If $y\in B(x,1/(a+b),L)$ then $d(\phi(y),\phi(x))=d(y,x)\le 1/(a+b)$ and $d(\phi(y),x')\le d(\phi(y),\phi(x))+d(\phi(x),x')\le 1/(a+b)+a\le 1/b$, as $a\le b\le 1/\sqrt 2$.
  If $z\in B(x'',1/(a+b),L'')$ then $d(\psi'(z),\psi'(x''))=d(z,x'')\le1/(a+b)$ and $d(\psi'(z),x')\le d(\psi'(z),\psi'(x''))+d(\psi'(x''),x')\le 1/(a+b)+b\le 1/a$ as  $a\le b\le 1/\sqrt 2$.
  This proves our claim.
\end{proof}

The reason why we require $\phi$ in the definition of $d$ to be an isometry is to have an easy proof of the triangle inequality.

We denote with $d'$ the combinatorial metric on $\Xi$, and with $d$ the metric on $\Omega$.
The restriction of $d$ to $\Xi$ induces a metric on $\Xi$, which, by the following lemma, is equivalent to the combinatorial one.

\begin{thm}\label{t:dequivalenttodponOmega}
  The discrete metric $d'$ on $\Xi$ is equivalent to the metric $d$ on $\Omega$ restricted to $\Xi$. More precisely  $d'\le d \le 18d'$.
\end{thm}
\begin{proof}
  If $d(\cisom{L,v},\cisom{L',v'})=1/n$ then they agree on a $d$-ball of radius $n$ with  $v$ mapped to $v'$ as the distance between the two vertices is at most $1/n$. But a $d$-ball of radius $n$ contains a $d'$-ball of radius $n$, so $d'(\cisom{L,v},$ $\cisom{L',v'})\le 1/n$. Therefore $d'\le d$.
  If $d'(\cisom{L,v},\cisom{L',v'})=1/n$ then there exists a cell-preserving map $\phi:B(v,n,L)\to L'$ and $\phi':B(v',n,L')\to L$ such that $\phi(v)=v'$ and $\phi'(v')=v$ and such that $\phi$, $\phi'$ restricted to each cell is an isometry. But a $d'$-ball of radius $n$ contains a $d$-ball of radius $n/3$, and $d(\phi(v),v')=0$, $d(v,\phi'(v'))=0$. By Lemma \ref{l:isometrywhenphirestricted}, the restriction of $\phi$ to the ball $B(v,n/18,L)$ is an isometry,  and so $d(\cisom{L,v},\cisom{L',v'})\le1/(n/18)=18/n$. So $d\le 18d'$.
\end{proof}

\begin{thm}\label{t:ContinuousHullCompact}
  The space $(\Omega,d)$ is compact.
\end{thm}
\begin{proof}
Let $\cisom{L_n,x_n}$ be a sequence in $\Omega$.
Let $\cisom{L_n,v_n}$ be a sequence in $\Omega$ such that $d(x_n,v_n)< 2$.
Since $\Xi$ is compact, there is a convergent subsequence $\cisom{L_{n_j},v_{n_j}}$ converging to some $\cisom{L,v}$ in the discrete metric $d'$, but since $d'$ is equivalent to the continuous metric $d$, the subsequence is also converging to $\cisom{L,v}$ in the continuous metric $d$.
Hence for a given $0<\varepsilon<1/3$ there is a sequence of isometry maps
$\phi_j:B(v_{n_j},1/\varepsilon,L_{n_j})\to (L,v)$ such that $d(\phi_j(v_{n_j}),v)<\varepsilon$ for all $j\ge N_\varepsilon$ for some $N_\varepsilon\in\N$.
By the triangle inequality, $d(\phi_j(x_{n_j}),v)< \varepsilon +2$.
So the sequence $\{\phi_j(x_{n_j})\}$ lies in a compact subset of $L$, and therefore there is a convergent subsequence $\phi_{j_k}(x_{n_{j_k}})\to x$.
Hence, for a given $0<\tilde\varepsilon<1/3$ we have $d(\phi_{j_k}(x_{n_{j_k}}),x)<\tilde\varepsilon$ for all $k\ge M_{\tilde\varepsilon}$ for some $M_{\tilde\varepsilon}$.
If $\tilde\varepsilon=\varepsilon/(1-2\varepsilon)$ then $B(x_{n_{j_k}},1/\tilde\varepsilon,L_{n_{j_k}})\subset B(v_{n_{j_k}},1/\varepsilon,L_{n_{j_k}})$ because for any $y\in B(x_{n_{j_k}},1/\tilde\varepsilon,L_{n_{j_k}})$ we have $d(y,v_{n_{j_k}})\le d(y,x_{n_{j_k}})+d(x_{n_{j_k}},v_{n_{j_k}})<1/\varepsilon-2+2<1/\varepsilon$.
Hence we can restrict $\phi_{{j_k}}$ to the ball $B(x_{n_{j_k}},1/\tilde\varepsilon,L_{n_{j_k}})$  for all $k\ge \max(N_\varepsilon,M_\varepsilon)$.
Since $d(x,v)\le d(x,\phi_{j_k}(x_{n_{j_k}}))+d(\phi_{j_k}(x_{n_{j_k}}),v)\le \tilde\varepsilon +\varepsilon +2$, $\phi'_{j_k}(x)$ exists.
Since $\phi_{j_k}(x_{n_{j_k}})\to x$, we have $d(\phi_{j_k}'(x),x_{n_{j_k}})\le \tilde\varepsilon$ for large $k$.
It follows that, $d(\cisom{L_{n_{j_k}},x_{n_{j_k}}},\cisom{L,x})\to 0$. Hence the space $(\Omega,d)$ is compact.
\end{proof}

\section{The  substitution map $\omega$ on $\Omega$}
We have seen that if $L$ is a combinatorial tiling locally isomorphic to $K$, then $\omega(L)$ is also a combinatorial tiling locally isomorphic to $K$.
Each tile $t$ of $L$ was replaced with a supertile $\omega(t)$.
Thus there is a homeomorphism from $\Laff$ to $\omega(L)_{\text{aff}}$ mapping each tile $t\in \Laff$ homeomorphically to the supertile $\omega(t)\subset\omega(L)_{\text{aff}}$.
There are of course many choices for the homeomorphism $\omega:t\to \omega(t)$, but we choose one that is affine. Among the affine homeomorphisms we choose  one that induces an invariant measure, which will be explained later.
\begin{figure}
  \centering
  \begin{overpic}[scale=.5]{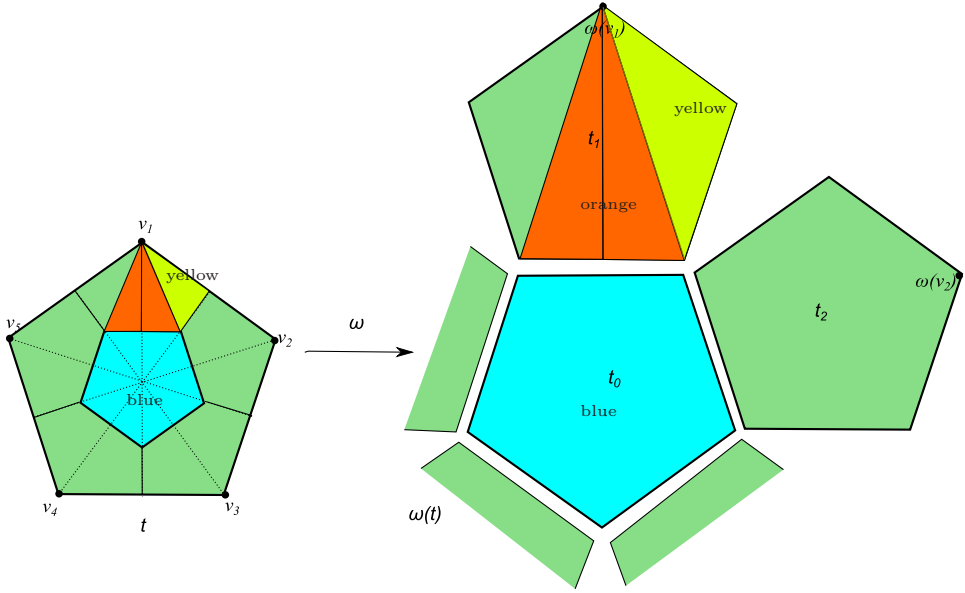}
  \put(215,144){\tiny orange}
  \put(250,180){\tiny yellow}
  \put(215,66){\tiny blue}
  \put(45,70){\tiny blue}
  \put(60,117){\tiny yellow}
  \end{overpic}
  \caption{The substitution map on a tile $t$. The yellow triangle is mapped onto the yellow triangle, and the orange one onto the orange one. The pentagon $\omega^{-1}(t_0)$ is a regular pentagon of side-length $s=\frac{1}{2} (\sqrt{5}-\sqrt{4-\sqrt{5}})=0.45$.}\label{f:omegacontinuoushull}
\end{figure}
\begin{defn}[\index{$\omega$ on a tile}$\omega$ on a tile]
We define the map  $\omega:t\to \omega(t)$ as  the homeomorphism that maps the Euclidean pentagon $t$ affinely to the half-dodecahedron $\omega(t)$ as in Figure \ref{f:omegacontinuoushull} carrying the decoration of $t$ to $\omega(t)$.
 In this figure, $t$ is partitioned into 6 smaller pentagons, such that $t$ (without decorations) has six rotations around the center of the pentagon and six reflections along the line going through a vertex and the midpoint of the opposite side; the blue pentagon is mapped to the blue pentagon by dilation by $2$; the orange and yellow triangles are mapped to the orange and yellow triangles linearly, respectively; by rotation and reflection of these triangles we define the map  on the rest of the tile $t$.
 \end{defn}

 \begin{defn}[\index{$\omega$ on $\Laff$}$\omega$ on $\Laff$]
 We define $\omega:\Laff\to\omega(L)_{\text{aff}}$ by defining it on each tile $t\in\Laff$ with the above affine map $\omega:t\to\omega(t)$.
 \end{defn}
 This map is well-defined on $\Laff$ because if $e$ is a common edge to two tiles $t,t'\in\Laff$ then $\omega(t\cap\{x\})=\omega(t'\cap\{x\})$ for all $x\in e$. See Figure \ref{f:omegacontinuoushullreflection}.
 The map $\omega:t\to\omega(t)$ induces a partition on $t$, namely $\{w^{-1}(c)\mid c \text{ a cell of } w(t) \}$, which is how we defined the map in the first place.
 This partition is a reflection of $\{w^{-1}(c)\mid c \text{ a cell of } w(t') \}$ along $e$, just as  $t$ is a reflection of $t'$ along $e$.
 See Figure \ref{f:omegacontinuoushullreflection}.
 The homeomorphism $\omega:\Laff\to \omega(L)_{\text{aff}}$ is locally affine but it is not cell-preserving, for each tile is mapped to six tiles.
\begin{figure}
  \centering
  \includegraphics[scale=.2]{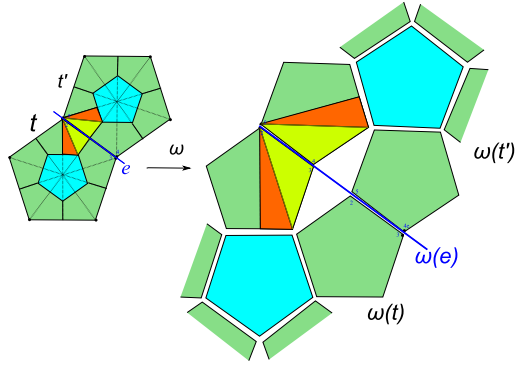}\\
  \caption{The substitution map $\omega:\Laff\to\omega(L)_{\text{aff}}$ is well defined on boundaries of tiles of $\Laff$.}\label{f:omegacontinuoushullreflection}
\end{figure}
\begin{defn}[\index{$\omega$ on $\Omega$}$\omega$ on $\Omega$]
We define the map $\omega:\Omega\to \Omega$ by
$$\omega(\cisom{L,x}):=\cisom{\omega(L),\omega(x)},$$
where $x\in\Laff$ and $\omega(x)\in\omega(L)_{\text{aff}}$.

\end{defn}
This map is well-defined for if $(L,x)$ is isomorphic to $(L',x')$ with $\phi$ denoting the cell-preserving isometric-on-each-cell isomorphism from the first one to the latter, then $\omega\circ\phi\circ\omega^{-1}:\omega(L)_{\text{aff}}\to \omega(L')_{\text{aff}}$ mapping $\omega(x)$ to $\omega(x')$ is as well cell-preserving isometric-on-each-cell isomorphism.

The map $\omega:\Omega\to\Omega$ we call it a substitution map (and not a subdivision map) because the size of the pentagons remains the same, namely unit regular pentagons.
Let $P,Q$ be patches of $K$. If $x\in P$, we call $(P,x)$ a pointed patch, or a patch with origin. We write  $(P,x)\sqsubseteq(Q,y)$  if there exists a cell-preserving isomorphism(onto its image) $\phi:P\to Q$ such that $\phi(x)=y$. By Proposition 1.29 in \cite{MRSdiscretehull} the isomorphism is unique.

\begin{figure}
  \centering
  \includegraphics[scale=.8]{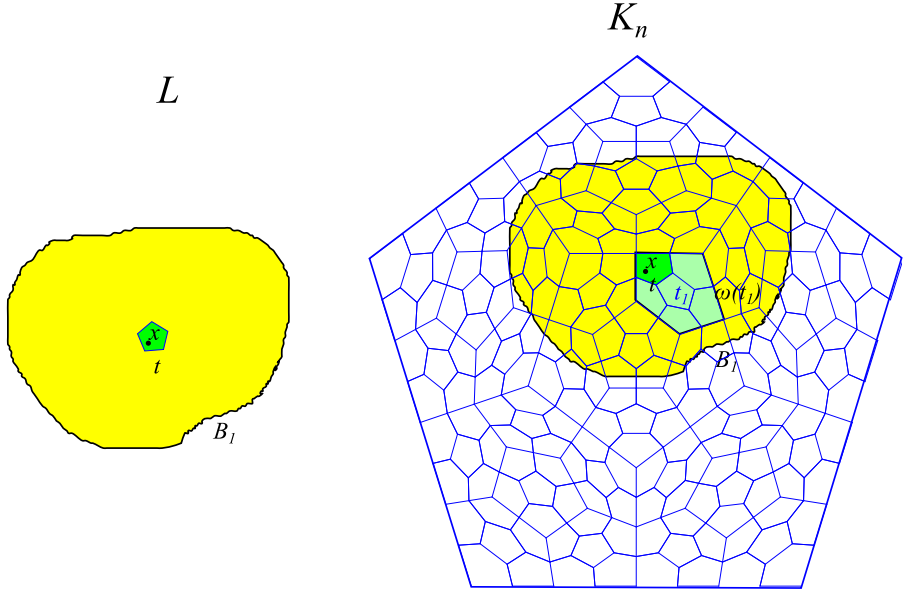}\\
  \caption{For each pointed tile $(t,x)\in L$ there is a pointed supertile $\omega(t_1,x_1)$ containing it.}\label{f:increasingsupertilesequence}
\end{figure}
\begin{pro}\label{p:increasingseqsupertiles}
  If $\cisom{L,x}\in \Omega$, then there exists an increasing sequence of pointed supertiles contained in $(L,x)$. More precisely,
  $$(t,x)\sqsubseteq\omega(t_1,x_1)\sqsubseteq\omega^2(t_2,x_2)\sqsubseteq\omega^3(t_3,x_3)\sqsubseteq\cdots\sqsubseteq(L,x).$$
\end{pro}
\begin{proof}
  Since $x\in L$, there is a tile $t\in L$ containing $x$, so $(t,x)\in (L,x)$.
  Let $r_1:=\max(\{diam(\omega(p))\mid p \text{ prototile}\})$ be the diameter of the largest supertile $\omega(p)$ among all prototiles $p$.
  Let $B_1$ be the ball in $L$ of center $x$ and radius $3r_1$.
  Since $B_1$ is in particular a patch, it is contained in a supertile. Hence $B_1$ is covered by supertiles $\omega(t')$, and since the ball $B_1$ is large enough, a supertile $\omega(t_1)$ contains $t$ and is contained in $L$. Hence $(t,x)\subset \omega(t_1,x_1)\subset (L,x)$. See Figure \ref{f:increasingsupertilesequence}.\\
  Let $r_2:=\max(\{diam(\omega^2(p))\mid p \text{ prototile}\})$ be the diameter of the largest supertile $\omega^2(p)$ among all prototiles $p$.
  Let $B_2$ be the ball in $L$ of center $x$ and radius $3r_2$.
  Since $B_2$ is in particular a patch, it is contained in a supertile. Hence $B_2$ is covered by supertiles of degree 2, and since the ball $B_2$ is large enough, a supertile $\omega^2(t_2)$ contains $\omega(t_1)$ and is contained in $L$. Hence $(t,x)\sqsubseteq \omega(t_1,x_1)\sqsubseteq \omega^2(t_2,x_2)\sqsubseteq (L,x)$.
  By induction, we obtain the increasing sequence.
\end{proof}
  \begin{figure}
    \centering
    \includegraphics[scale=.9,viewport=100 130 300 300,clip]{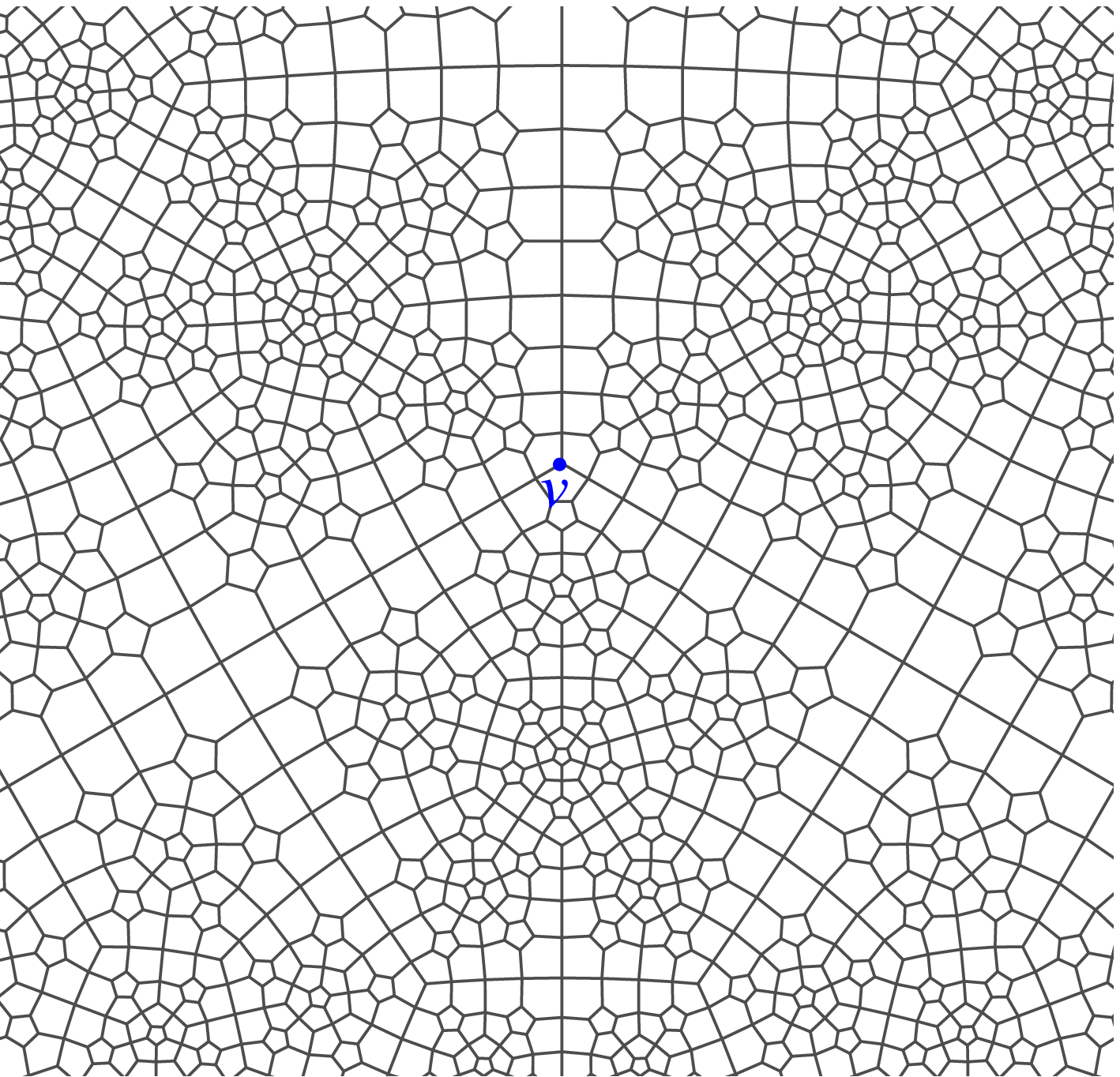}
    \caption{A tiling with central vertex $v$, which is locally isomorphic to $K$.}
    \label{f:fixedpointsofw03}
  \end{figure}
The union $\union_n (\omega^n(t_n,x_n))$ is contained in  $(L,x)$,  but it is not necessarily the whole $(L,x)$.
For example, the pointed tiling $(L,v)$ shown in Figure \ref{f:fixedpointsofw03} has the increasing sequence $(t_0,v_0)\sqsubseteq \omega(t_1,v_1)\sqsubseteq \omega^2(t_2,v_2)\sqsubseteq \cdots\sqsubseteq (L,v)$, where $t_i$ is a prototile that has two vertices of degree $4$ and $v_i$ is the vertex of degree 3 between these two vertices (the vertex $v_i$ is shown in this figure as $v$). Their union is just a third of $(L,v)$.

By definition, for each tiling $\cisom{L,x}\in \Omega$, $L$ is locally isomorphic to $K$.
Informally, this means that every patch of $L$ is contained in some supertile $K_n$.
A natural question arises. Is every supertile $K_n$ in $L$? The following lemma shows that the answer is yes.

\begin{thm}
  If $\cisom{L,x}\in \Omega$, then any supertile $K_n$ is a subcomplex of $L$. In particular $K$ is locally isomorphic to $L$.
\end{thm}
\begin{proof}
  By Proposition \ref{p:increasingseqsupertiles}, there is an increasing sequence of supertiles in $(L,x)$.
  Since we have a finite number of prototiles, there is one prototile that repeats infinitely many times, say the prototile $t$, i.e. $\omega^n(t)$ appears somewhere in $L$ for any $n$.
  Since the subdivision rule is primitive, there is an integer $k$ such that $\omega^k(t)$ contains all the prototiles.
  Hence $\omega^{kn}(t)$ contains any supertile of degree $n$.
\end{proof}

\begin{lem}\label{l:OmegaInjective}
The substitution map $\omega:\Omega\to \Omega$ on the continuous hull $\Omega$ is injective.
\end{lem}
\begin{proof}
  Suppose $\omega(\cisom{L,x})=\omega(\cisom{L',x'})$ then $\omega(L,x)$ and $\omega(L',x')$ are isomorphic, and since the isomorphism is cell-preserving, $\omega(L,v)\cong\omega(L',v')$ for some vertices $v,v'$.  Since $\omega:\Xi\to\Xi$ is injective, $(L,v)\cong(L',v')$.
  Since this isomorphism induces a cell-preserving isometric-on-each-cell isomorphism between  $\Laff $ and $\Laff'$ and $\omega$ on a tile is a homeomorphism,   $(\Laff,x)\cong(\Laff',x')$.
\end{proof}

\subsection{On the partition of a regular pentagon}

\begin{lem}\label{l:matrixfortriangles}
  If $x$ and $y$ are nonzero nonparallel vectors of $\R^2$ positioned at the origin, then we obtain the triangle $T_1$ with vertices at $0,x,y$.
  Let $x',y'$ be another two nonzero nonparallel vectors of $\R^2$ positioned at the origin,  and let $T_2$ be the triangle with vertices at $0,x',y'$.
  Define the $2\times 2$ matrix $A:=[x' y'][x y]^{-1}$, where $x,y,x',y'$ are  column vectors.
  Then the linear map $f:T_1\to T_2$ defined by $f(v)=Av$ is bijective.
  Moreover, $||Av||\le ||A||\cdot ||v||$, where $||A||=\sqrt{\lambda_{\max}(A^*A)}$, with $\lambda_{\max}(A^*A)$ being the maximum eigenvalue of the matrix $A^*A$, and $A^*$ the conjugate transpose of the matrix $A$.
\end{lem}
\begin{proof}
  By definition of $A$ we have $x'=Ax$ and $y'=Ay$. Hence $s x'+ t y'= s Ax + t Ay=A( s x+ t y)$ for any $s,t\in\R$.
  Restrict $s,t$ to the triangle $P:=\{(s,t)\mid 0\le t\le s \le 1\}$.  Then $T_1=[x\,\, y]P$ and $T_2=[x'\,\, y']P=A[x\,\, y]P=AT_1$.
  Since the matrix $A$ is invertible, the map $f$ is bijective.
  The last statement in the lemma follows from the fact that the Euclidean norm induces the matrix norm $||A||=\max\{\frac{||Av||}{||v||}\mid v\ne 0\}=\sqrt{\lambda_{\max}(A^*A)}$.
\end{proof}
\begin{figure}
\centering
  \begin{overpic}[scale=.7]{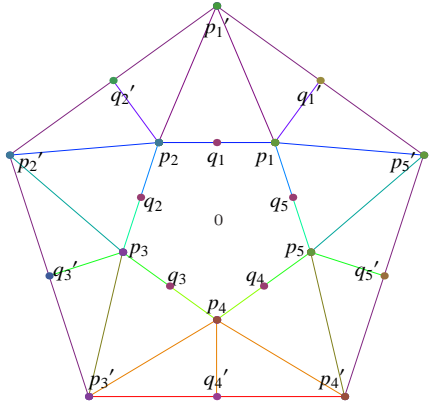}
    \put(90,90){\tiny{0}}
  \end{overpic}
  \caption{Partition of a regular pentagon of side length $1$.}\label{f:numbersofomega}
\end{figure}
The sizes of the colored  regions of the pentagon $t$ in Figure \ref{f:omegacontinuoushull} can be computed from the coordinates $p_1,q_1,p_1',q_1'$ defined below and shown in Figure \ref{f:numbersofomega}. The rest of the coordinates can be computed by rotating them by $2\pi/5$; for instance $p_2$ is the rotation of $p_1$ by $2\pi/5$ around the origin.
\begin{eqnarray*}
\begin{array}{l}
 p_1:=(\frac{1}{4} (\sqrt{5}-\sqrt{4-\sqrt{5}}),\frac{1}{4} (-\sqrt{2+\frac{3}{\sqrt{5}}}+\sqrt{5+2 \sqrt{5}})) \\
 q_1:=(0,\frac{1}{4} (-\sqrt{2+\frac{3}{\sqrt{5}}}+\sqrt{5+2 \sqrt{5}})) \\
 p_1':=(0,\sqrt{\frac{1}{10} (5+\sqrt{5})}) \\
 q_1':=(\frac{1}{8} (1+\sqrt{5}),\frac{1}{8} \sqrt{10+\frac{22}{\sqrt{5}}}).
\end{array}
\end{eqnarray*}
By Lemma \ref{l:matrixfortriangles}, the matrices mapping the small blue, orange, yellow regions  in Figure \ref{f:omegacontinuoushull} to the large regions of same color, respectively, are
\begin{eqnarray*}
&& M_0:=\frac{||p_1'||}{||p_1||}RotationMatrix(\pi/5)=\left(
\begin{array}{cc}
 1.7821 & -1.29477 \\
 1.29477 & 1.7821 \\
\end{array}
\right)\\
&&M_1:=[p_4'-p_3',p_1'-p_3'][p_1-p_2,p_1'-p_2]^{-1}=\left(
\begin{array}{cc}
 2.2028 & 0 \\
 0 & 2.85906 \\
\end{array}
\right)\\
&&M_2:=[p_5'-p_4',p_1'-p_4'][q_1'-p_1,p_1'-p_1]^{-1}=\left(
\begin{array}{cc}
 1.91042 & -0.123303 \\
 0.899856 & 3.23855 \\
\end{array}
\right).
\end{eqnarray*}
The matrix norms induced by the Euclidean norm are
$$||M_0||=2.2028,\quad ||M_1||=2.85906,\quad ||M_2||=3.39406,$$
$$||M_0^{-1}||=||M_1^{-1}||=0.453968,\quad ||M_2^{-1}||=0.538918.$$
The determinants are  $\det M_0=4.85231$,  $\det M_1=   \det M_2=6.29792$.

\begin{lem}\label{l:domegacomparison}
  If $x,y\in \Laff$ then $$d(\omega(x),\omega(y))\le3.40\, d(x,y)  \text{ and } d(x,y)\le 0.54\, d(\omega(x),\omega(y)).$$
\end{lem}
\begin{proof}
  We start by showing the first inequality. Consider the colored regions shown in Figure \ref{f:omegacontinuoushull}.
  If $x$ and $x'$ are in the small blue pentagon then $\omega(x)$ and $\omega(y)$ both lie in the large  blue pentagon and thus $d(\omega(x),\omega(y))=||M_0|| d(x,y)$ because the large blue pentagon is dilation of the small one by $||M_0||$.
  If $x$ and $x'$ are in the small orange triangle of $t$ then $\omega(x)$ and $\omega(y)$ both lie in the large orange triangle of $\omega(t)$.
  By Lemma \ref{l:matrixfortriangles},  $d(\omega(x),\omega(y))\le ||M_1||d(x,y)$.
  Similarly, for two points $x,y$ in the small yellow triangle, we have $d(\omega(x),\omega(y))\le ||M_2||d(x,y)$.
  Let $x,y\in \Laff$ and let $\gamma_1\cup\cdots\cup\gamma_n$ be a path of shortest length with endpoints $x,y$, such that each line-segment $\gamma_i$ lies in either a blue, orange, or yellow region. Then
  \begin{eqnarray*}
   d(\omega(x),\omega(y))&\le& length(\omega(\gamma_1))+\cdots+length(\omega(\gamma_n))\\
   &\le& max(||M_0||,||M_1||,||M_2||)d(x,y).
  \end{eqnarray*}
  Similarly, if $x',y'\in \omega(L)_{\text{aff}}$ then
  $$d(\omega^{-1}(x'),\omega^{-1}(y'))\le \max(||M_0^{-1}||,||M_1^{-1}||,||M_2^{-1}||) d(x',y'),$$ as a path of shortest length is composed of joined line-segments and linear maps send lines to lines.
\end{proof}

\subsection{More properties of the substitution map $\omega$ on $\Omega$}
\begin{lem}
  If $\cisom{L,x},\cisom{L',x'}\in\Omega$ then $$d(\omega(\cisom{L,x}),\omega(\cisom{L',x'}))\le 20.4d(\cisom{L,x},\cisom{L',x'}). $$
\end{lem}
\begin{proof}
  Let $d(\cisom{L,x},\cisom{L',x'})=\varepsilon$.
  Then there exist embeddings
  $\phi:B(x,\varepsilon^{-1},L)\to L'$,\,\,\,\,\,
    $\phi':B(x',\varepsilon^{-1},L')\to L$,
    such that $d(\phi(x),x')<\varepsilon$, $d(\phi'(x'),x)<\varepsilon$.
    Thus $d(\omega(\phi(x)),\omega(x'))<3.4\varepsilon$, $d(\omega(\phi'(x')),\omega(x))<3.4\varepsilon$.
    We claim that $B(\omega(x),\frac{\varepsilon^{-1}}{3.4},\omega(L))\subset \omega(B(x,\varepsilon^{-1},L))$.\\
    Let $\omega(y)\in B(\omega(x),\frac{\varepsilon^{-1}}{3.4},\omega(L))$.
        Then $\tfrac{\varepsilon^{-1}}{3.4}\ge d(\omega(x),\omega(y))\ge \tfrac1{0.54} d(x,y)$.
    So  $d(x,y)< \varepsilon^{-1}$.
    Similarly, $B(\omega(x'),\frac{\varepsilon^{-1}}{3.4},\omega(L'))\subset \omega(B(x',\varepsilon^{-1},L'))$.\\
    Define $\psi:B(\omega(x),\frac{\varepsilon^{-1}}{3.4},\omega(L))\to \omega(L')$ by
    $\psi(y):=\omega(\phi(\omega^{-1}(y)))$.
    Define $\psi':B(\omega(x'),\frac{\varepsilon^{-1}}{3.4},\omega(L'))\to \omega(L)$ by
    $\psi'(y'):=\omega(\phi'(\omega^{-1}(y')))$.
    Since $\omega:\Laff\to \omega(L)_{\text{aff}}$ is a homeomorphism, $\psi$ and $\psi'$ are continuous and cell-preserving and the restriction to each cell is an isometry.
    By Lemma \ref{l:isometrywhenphirestricted} $\psi$ restricted to $B(\omega(x),\frac 16\cdot\frac{\varepsilon^{-1}}{3.4},\omega(L))$  and
    $\psi'$ restricted  to $B(\omega(x'),\frac16\cdot\frac{\varepsilon^{-1}}{3.4}$ $,\omega(L'))$  are isometries.
\end{proof}

\begin{cor}\label{c:Omegacontinuous}
  The substitution map $\omega:\Omega\to\Omega$ on the continuous hull $\Omega$ is continuous.
\end{cor}
\begin{proof}
By the lemma, $\omega:\Omega\to\Omega$ is Lipschitz continuous and so uniformly continuous. In particular, it is continuous.
\end{proof}

In \cite{MRSdiscretehull} we showed that $\omega$ on the discrete space $\Xi$ is not surjective. However, it is surjective on the continuous hull as the following lemma shows.
\begin{lem}\label{l:Omegasurjective}
The substitution map $\omega:\Omega\to \Omega$ on the continuous hull $\Omega$ is surjective.
\end{lem}
\begin{proof}
Let $\cisom{L,x}\in \Omega$. Let $B_n:=B(x,n,L)$ be the ball in $L$ of radius $n$ and center $x$.
By definition of $L$, the ball $B_n$ lies in a supertile $\omega^{m_n}(t_n)\subset K$. Defining $P_n:=\omega^{m_n-1}(t_n)$ we have $(B_n,x)\sqsubseteq (\omega(P_n),x)$. Let $x_n:=\omega^{-1}(x)$ be the point in $K$ such that $\omega(P_n,x_n)=(\omega(P_n),x)$.
Since $\Omega$ is compact, there is a subsequence $\{\cisom{K,x_{n_k}}\}_{k\in\N}$ converging to some element $\cisom{L',x'}\in\Omega$.
By construction, $d(\omega(\cisom{K,x_{n_k}}),\cisom{L,x})\le 1/{n_k}$.
Since $\omega$ is continuous,
$$\omega(\cisom{L',x'})=\omega(\lim_{k\to\infty}\cisom{K,x_{n_k}})=\lim_{k\to\infty}\omega(\cisom{K,x_{n_k}})=\cisom{L,x}.$$
\end{proof}

\begin{thm}\label{t:OmegaHomeomorphism}
  The substitution map $\omega:\Omega\to \Omega$ on the continuous hull $\Omega$ is a homeomorphism.
\end{thm}
\begin{proof}
  By Theorem \ref{t:ContinuousHullCompact}, Corollary  \ref{c:Omegacontinuous}, Lemma \ref{l:OmegaInjective} and Lemma \ref{l:Omegasurjective},
  the substitution map $\omega:\Omega\to \Omega$ is a continuous bijective map on the compact space $\Omega$. Hence $\omega$ is a homeomorphism.
\end{proof}

\section{$\Xi$ transversal to $\Omega$.}

\begin{defn}[full transversal]
  Let $R$ be an equivalence relation on a space $X$. Then $Y\subset X$ is a transversal to $X$ if $[x]_R\cap Y$ is countable for any $x\in X$.
  It is said to be a \emph{full transversal} if  $[x]_R\cap Y\ne \emptyset$ for each $x\in X$.
\end{defn}
It follows that $R_Y:=\{(x,y)\subset Y\times Y\mid (x,y)\in R\}$  is an equivalence relation on $Y$, and if $Y$ is a full transversal then $Y$ contains at least one representative of each equivalence class $[x]_R$. Moreover, $[y]_R\cap [y]_{R_Y}=[y]_{R_Y}$ for any $y\in Y$.

\begin{defn}[equivalence relation $R$ on $\Omega$]\index{$R$ (equivalence relation)}
  Define the equivalence relation $R$ on $\Omega$ by
  $$R:=\{(\cisom{L,x},\cisom{L,x'})\in\Omega\times \Omega\mid x,x'\in \Laff.\}.$$
\end{defn}
\begin{lem}\label{l:XifulltransversalonOmega}
  The discrete hull $\Xi$ is a full transversal on the continuous hull $\Omega$ relative to $R$, and $R_{\Xi}=R'$.
\end{lem}
\begin{proof}
  By definition $[\cisom{L,x}]_R\cap \Xi=\{[\cisom{L,x}]_R\mid x\in \Laff \text{ and } x\in L^{(0)}\}= \{\cisom{L,v}\mid v\in L^{(0)}\}=[\cisom{L,v}]_{R'}$, which is a countable set (as a geometric realization of $L$ is a tiling of the plane and so it has countably many tiles and each tile has 5 vertices).  The intersection is nonempty as it is in bijective correspondence with the vertices of $L$. So we have a full transversal.
\end{proof}

\begin{pro}
The discrete hull $\Xi$ is a closed compact subset of $\Omega$.
\end{pro}
\begin{proof}
Since the metric $d$ on $\Omega$ restricted to $\Xi$ is equivalent to the metric $d'$ on $\Xi$ (Theorem \ref{t:dequivalenttodponOmega}) and $\Xi$ is compact (Lemma 2.11 in \cite{MRSdiscretehull}),  the space $\Xi$ is a compact subset of $\Omega$. Since $\Omega$ is Hausdorff, $\Xi$ is closed.
\end{proof}

\begin{description}
  \item[\emph{\textbf{Question}.}] It is an open question whether the $C^*$-algebra $C^*(R)$ exists. For that it needs a topology and a Haar system.
\end{description}


\subsection*{Acknowledgments.}
The results of this paper were obtained during my Ph.D. studies at University of Copenhagen. I would like to express deep gratitude to my supervisor Erik Christensen and Ian F. Putnam whose guidance and support were crucial for the successful completion of this project.



\begin{thebibliography}{99}



\bibitem{AConnes}
  A. Connes,
  \emph{ Non-commutative Geometry.}
 Academic Press,
 San Diego  (1994).

\bibitem{FloydFiniteSubdivisionRules01}
   J. W. Cannon, W. J. Floyd, and W. R. Parry,
   \emph{Finite subdivision rules.}
   http://www.math.vt.edu/people/floyd/research/papers/fsr.pdf,
   (2001).


\bibitem{Hatcher02}
  Allen Hatcher,
  \emph{ Algebraic Topology.}
  Cambridge University Press,
  2002.

\bibitem{TopologyJanich84}
   K. J\"anich, S. Levy,
   \emph{Topology.}
   Springer-Verlag New York Inc.,
   (1984).




\bibitem{May99}
  J. P. May,
  \emph{ A concise course in Algebraic Topology.}
  The Univeristy of Chicago Press.
  Chicago and London,
  1999.


\bibitem{Mozes89}
   Shahar Mozes,
   \emph{Tilings, substitution systems and dynamical systems generated by them.}
   Journal D'analyse Mathematique,
   (1989).


  \bibitem{mr82}
  Paul S. Muhly, Jean N. Renault,
  \emph{$C^*$-algebras of multivariable Wiener-Hopf operators.}\\
   American Mathematical Society,
   (1982).


\bibitem{mrw87}
  Paul S. Muhly, Jean N. Renault, Dana P. Williams
  \emph{Equivalence and isomorphism for groupoid $C^*$-algebras.}
   J. Operator Theory,
   (1987) 3-22.


\bibitem{FinDimApproxCstaralgebrasBrownOzawa}
   Nathanial Patrick Brown, Narutaka Ozawa,
  \emph{ C*-Algebras and Finite-Dimensional Approximations.}
  American Mathematical Society,
  2008.

\bibitem{Putnam00OrderedKtheory}
 Ian F. Putnam,
 \emph{The ordered $K$-theory of $C^*$-algebras associated with substitution tilings.}
 Commun. Math. Phys. 214,
 (2000) 593-605.

\bibitem{PutnamBible95}
  Jared E. Anderson and Ian F. Putnam.
  \emph{Topological Invariants for Substitution Tilings and their Associated $C^*$-algebras.}
  Department of Mathematics and Statistics, University of Victoria, Victoria  B.C. Canada.
  (1995) 1-45.



\bibitem{PutnamCstarKtheory00}
  Johannes Kellendonk and Ian F. Putnam,
  \emph{Tilings, $C^*$-algebras and $K$-theory.}
  Directions in mathematical quasicrystals, CRM Monogr. Ser., 13, Amer. Math. Soc., Provicence, RI
  (2000) 177-206.



\bibitem{Putnametalenotes}
   Ian F. Putnam,
  \emph{ Orbit equivalence of Cantor minimal systems:Kyoto Winter School 2011.}\\
    http://www.math.uvic.ca/faculty/putnam/r/Kyoto\_2011\_main.pdf\\
  2011.


\bibitem{putnametalenotesreal}
  Jason Peebles, Ian F. Putnam, Ian Zwiers
  \emph{Minimal Dynamical Systems on the Cantor Set.}\\
   Lecture notes,
   (2011).

   \bibitem{MRSnonFLCpentTiling}
   Maria Ramirez-Solano,
  \emph{ A non FLC regular pentagonal tiling of the plane.}\\
   arXiv:1303.2000,
  2013.

\bibitem{MRSdiscretehull}
   Maria Ramirez-Solano,
  \emph{ Construction of the discrete hull for the combinatorics of a  regular pentagonal tiling of the plane.}
   arXiv:1303.5375,
  2013.


\bibitem{Renault80}
  Jean Renault,
  \emph{ A groupoid approach to $C^*$-algebras.}
  Lecture Notes in Mathematics, No.793,
  Springer-Verlag, Berlin-New York,
  1980.


\bibitem{RenaultCstarAlgDynSystem}
  Jean Renault,
  \emph{$C^*$-algebras and dynamical systems.}\\
 http://www.univ-orleans.fr/mapmo/membres/renault/books/IMPA\_09.pdf,
   (2009).


\bibitem{MeasureandintegrationLeonard2009}
 Leonard F. Richardson,
  \emph{ Measure and Integration.}
 Wiley,
  (2009).


\bibitem{Rob91}
  E. Arthur Robinson, Jr.,
  \emph{ Symbolic Dynamics and Tilings of $\R^d$.}
  George Washington University.
  Washington DC.
  1991.


\bibitem{Sadun08}
  Lorenzo Sadun,
  \emph{ Topology of Tiling Spaces.}
  University Lecture Series Vol. 46,
  Providence, Rhode Island,
  2008.

\bibitem{SadunTilingSpacesAreCSFB}
   Lorenzo Sadun, R. F. Williams,
  \emph{ Tiling Spaces Are Cantor Set Fiber Bundles.}\\
   http://arxiv.org/pdf/math/0105125.pdf\\
  2001.


\bibitem{Sol98}
  B. Solomyak,
  \emph{ Nonperidicity implies unique composition of self-similar translationally-finite tilings.}
  Disc. Comp. Geom.
  1998.

\bibitem{StephensonBowers97}
  Philip L. Bowers and Kenneth Stephenson,
  \emph{A "regular" pentagonal tiling of the plane.}
  Conformal geometry and dynamics.
  An electronic journal of the American Mathematical Society,
  (1997) 58-86.














\end{thebibliography}
\end{document}